\documentclass[reqno,a4paper,11pt]{amsart}
\usepackage[utf8]{inputenc}

\usepackage[english]{babel}

\usepackage[utf8]{inputenc}
\usepackage[T1]{fontenc}
\usepackage[normalem]{ulem}

\usepackage{amssymb}
\usepackage{amsmath,amsfonts}
\usepackage[tikz]{bclogo}
\usetikzlibrary{arrows, arrows.meta}
\usepackage[top=3cm,bottom=3cm,left=2.5cm,right=2.5cm]{geometry}

\usepackage{lmodern}
\usepackage{textcomp}
\usepackage{mathtools, bm}
\usepackage{amssymb, bm}
\usepackage[unq]{unique}
\usepackage{enumerate}
\usepackage{comment}
\usepackage[breaklinks]{hyperref}

\usepackage[numbers]{natbib}  

\usepackage{graphicx} 

\usepackage{amsthm}

\usepackage{tikz}
\usepackage{caption}
\usepackage{subcaption}

\usepackage{cleveref}

\usepackage{bbm}

\usepackage{enumitem}
\usepackage{thmtools}
\usepackage{thm-restate}

\newtheorem{theorem}{Theorem}[section]
\newtheorem{lemma}[theorem]{Lemma}

\theoremstyle{definition}

\newtheorem{definition}[theorem]{Definition}

\DeclareMathOperator{\rank}{rank}
\newcommand{\Rd}{\mathcal R_d}

\title{Cliques in minimally globally rigid graphs}
\date{}

\author{Julien Portier$^*$}

\begin{document}

\renewcommand{\thefootnote}{\fnsymbol{footnote}}
\footnotetext[1]{Institute of Mathematics, EPFL, Lausanne, Switzerland. \texttt{julien.portier@epfl.ch}}
\renewcommand{\thefootnote}{\arabic{footnote}}

\maketitle

\begin{abstract}
    We show that every minimally generically globally rigid graph in
$\mathbb R^d$ which contains a subgraph isomorphic to $K_{d+2}$ is itself
isomorphic to $K_{d+2}$, confirming a conjecture by Garamv{\"o}lgyi, Jackson, and Jord{\'a}n. The proof is entirely generated by ChatGPT 5.5.
\end{abstract}

\section{Introduction}

\begin{definition}[Generic global rigidity]
Let $G=(V,E)$ be a graph. A framework $p:V\to \mathbb R^d$ is said to be
\emph{globally rigid} in $\mathbb R^d$ if, for every framework
$q:V\to \mathbb R^d$ satisfying
\[
\|p(u)-p(v)\|=\|q(u)-q(v)\|
\qquad
\text{for every }uv\in E,
\]
there exists an isometry of $\mathbb R^d$ mapping $p(v)$ to $q(v)$ for
every $v\in V$.
The graph $G$ is said to be \emph{generically globally rigid} in
$\mathbb R^d$ if every generic framework $p:V\to \mathbb R^d$ is globally
rigid in $\mathbb R^d$.
\end{definition}

Our main result is the following, which confirms a conjecture by Garamv{\"o}lgyi, Jackson, and Jord{\'a}n (Conjecture 6.3 in \cite{garamvolgyi2025sparsity}).

\begin{theorem}
\label{thm:main}
Let $G=(V,E)$ be an edge-minimal generically globally rigid graph in
$\mathbb R^d$, and suppose that $|V|\ge d+3$.
Then $G$ contains no subgraph isomorphic to $K_{d+2}$.
\end{theorem}

\noindent\textbf{Comment on the use of AI.}
The proof of \Cref{thm:main} was generated by
ChatGPT 5.5, accessed through a Plus subscription. Before consulting
ChatGPT, the author was already aware of the following natural strategy to prove \Cref{thm:main}:
take a stress matrix of rank $r=n-d-1$ on the whole graph, take the simplex stress supported on the clique $K_{d+2}$, and form a linear combination of the two in order to obtain a stress matrix which simultaneously 1) vanishes on a chosen edge of the clique, and 2) still has rank $r$.
However, it was unclear how to satisfy these two requirements at the
same time. The author suspected that some algebraic trick might resolve
this difficulty, which was the main reason for posing the problem to
ChatGPT.

The author emphasizes that this intuition was not included in the prompt
given to ChatGPT. The model independently identified this strategy, and produced the subsequent proof.
The author then checked the proof and edited it for clarity and exposition, and takes responsibility
for the final mathematical content of the paper. \\

\noindent\textbf{Acknowledgements.}
The author would like to thank D{\'a}niel Garamv{\"o}lgyi and Tibor
Jord{\'a}n for interesting discussions about the subject. Moreover, after
the writing of this manuscript, Tibor Jord{\'a}n found another short proof
of \Cref{thm:main}, relying almost exclusively on Lemma~3.2 and
Corollary~4.2 of \cite{jordan2024globally}.

\section{Preliminairies}

We start with the following definition of a stress matrix.

\begin{definition}[Equilibrium stress matrix]
Let $G=(V,E)$ be a graph, and let $p:V\to \mathbb R^d$ be a framework.
An \emph{equilibrium stress matrix} of $(G,p)$ is a symmetric matrix $\Omega$
indexed by $V$ such that:
\begin{enumerate}
    \item $\Omega_{uv}=0$ for all distinct $u,v\in V$ with $uv\notin E$;
    \item $\displaystyle \sum_{v\in V}\Omega_{uv}=0$ for all $u\in V$;
    \item $\displaystyle \sum_{v\in V}\Omega_{uv}p(v)=0$ for all $u\in V$.
\end{enumerate}
\end{definition}

To avoid clutter, we will usually drop the word ``equilibrium'' and
refer simply to stress matrices. 
We will use the following result, where one direction of the equivalence was shown by Connelly (Theorem 1.3 in \cite{connelly2005generic}), and the other by Gortler, Healy, and Thurston (Theorem 1.14 in \cite{GHT2010}).

\begin{theorem}[Stress-rank characterization of generic global rigidity]
\label{thm:stress-rank}
Let $G=(V,E)$ be a graph with $n\ge d+2$ vertices, and let
$p:V\to \mathbb R^d$ be a generic framework. Then $G$ is generically
globally rigid in $\mathbb R^d$ if and only if $(G,p)$ admits an
equilibrium stress matrix $\Omega$ such that
\[
\rank \Omega = n-d-1.
\]
\end{theorem}

Next we recall the following definitions. Given a graph $G=(V,E)$, the $d$-dimensional rigidity matrix $R_d(G,p)$ of a generic framework $p:V\to \mathbb R^d$ is the matrix with one row for each edge $uv\in E$ and $d$ columns for each vertex. In the row corresponding to $uv$, the block of columns corresponding to $u$ is $p(u)-p(v)$, the block corresponding to $v$ is $p(v)-p(u)$, and all other entries are zero.
Let $\mathcal R_d(G)$ denote the matroid on $E(G)$ represented by the
rows of the $d$-dimensional rigidity matrix $R_d(G,p)$.
Note that this definition does not depend on the generic framework $p$.

\begin{definition}[$\mathcal R_d$-circuit]
Let $G$ be a graph. A set of edges $C\subseteq E(G)$ is called an
$\mathcal R_d$-\emph{circuit} if $C$ is minimally dependent in
$\mathcal R_d(G)$; equivalently, the rows of the generic $d$-dimensional
rigidity matrix indexed by $C$ are linearly dependent, while the rows
indexed by every proper subset of $C$ are linearly independent.
\end{definition}

\begin{definition}[$\Rd$-connected]
We say that a graph $G$ is $\Rd$-connected if for every pair of edges
$e,f\in E(G)$ there is an $\Rd$-circuit $C\subseteq E(G)$ such that
$e,f\in C$.
\end{definition}

We recall the following result by Garamv{\"o}lgyi, Gortler, and Jord{\'a}n (Theorem 3.5 in \cite{GGJ2022}).

\begin{theorem}
\label{thm:Rd-connected}
Let $G$ be a graph with $n\ge d+2$ vertices. If $G$ is generically
globally rigid in $\mathbb R^d$, then $G$ is $\Rd$-connected.
\end{theorem}

We now show the following simple lemma.

\begin{lemma}[Full-rank stresses are dense]
\label{lem:full-rank-stresses-dense}
Let $(G,p)$ be a generic framework in $\mathbb R^d$ on $n$ vertices, and
let $\mathcal S(G,p)$ denote the vector space of equilibrium stress
matrices of $(G,p)$. Suppose that $\mathcal S(G,p)$ contains an equilibrium stress
matrix of rank $r=n-d-1$.
Then the set
\[
\mathcal S_r(G,p)
=
\{\Omega\in \mathcal S(G,p): \rank \Omega=r\}
\]
is dense in $\mathcal S(G,p)$, with respect to the usual Euclidean
topology on the finite-dimensional vector space $\mathcal S(G,p)$.
\end{lemma}

\begin{proof}
It is standard that every stress matrix of a generic framework on $n$
vertices in $\mathbb R^d$ has rank at most $r=n-d-1$.

Let $\Omega_0\in \mathcal S(G,p)$ with $\rank \Omega_0=r$. Then $\Omega_0$ has a $r \times r$ submatrix $\Omega^*$ such that $\det(\Omega^*) \neq 0$. Let $I$ and $J$ be respectively the set of rows and columns of $\Omega^*$.

Let $\Omega\in\mathcal S(G,p)$ be arbitrary. Consider
\[
f(t)
=
\det\bigl(\Omega[I,J]+t\Omega^*\bigr),
\]
where $\Omega[I,J]$ is the submatrix of $\Omega$ restricted to set of rows $I$ and set of columns $J$.
Since the leading coefficient of $f(t)$ is $\det(\Omega^*)\ne 0$, then the polynomial $f$ is not identically zero. Hence it has only finitely
many zeros. Therefore, for arbitrarily small values of $t$, we have $f(t)\ne 0$, and thus $\rank(\Omega+t\Omega_0) = r$.
As it is easy to see that $\Omega+t\Omega_0 \in \mathcal S(G,p)$ for every $t$, this concludes the proof.
\end{proof}

\section{Proof of \Cref{thm:main}}

We start with the following lemma.

\begin{lemma}
\label{lem:clique-proportionality-full-rank}
Let $G=(V,E)$ be minimally generically globally rigid in $\mathbb R^d$,
and suppose $|V|=n\ge d+3$, and let $r=n-d-1$.
Let $p:V\to \mathbb R^d$ be a generic framework of $G$. Suppose that
$X\subseteq V$ spans a copy of $K_{d+2}$ in $G$. Then there exist
nonzero reals $a_i$, $i\in X$, such that for every stress matrix
$\Omega$ of $(G,p)$ with $\rank \Omega=r$, there exists $\lambda=\lambda(\Omega) \in \mathbb{R}$ such that
$\Omega_{ij}=\lambda a_i a_j$ for every  $ij\in E(G[X])$.
\end{lemma}

\begin{proof}
Since $|X|=d+2$ and $p$ is generic, the points $\{p_i:i\in X\}$ have a
unique affine dependence, up to scaling. Thus there exist reals
$a_i$, $i\in X$, not all zero, such that
\[
\sum_{i\in X} a_i p_i=0,
\qquad
\sum_{i\in X} a_i=0.
\]
Moreover, by genericity, every $a_i$ is nonzero. Extend $a$ to a vector
in $\mathbb R^V$ by setting $a_i=0$ for $i\notin X$, and define $A=aa^T$.
We now claim that $A$ is a stress matrix supported on the clique $G[X]$. Indeed,
\[
A\mathbf 1=a(a^T\mathbf 1)=0
\qquad\text{and}\qquad
AP=a(a^TP)=0,
\]
where $P$ is the configuration matrix of $p$, namely the $|V|\times d$
matrix whose row indexed by $v\in V$ is $p(v)^T$. Also, for every clique
edge $ij\in E(G[X])$, we have $A_{ij}=a_i a_j\ne 0$.

Now let $\Omega$ be a stress matrix of $(G,p)$ with rank $r$. Fix a
clique edge $ij\in E(G[X])$, and define
\[
c_{ij}:=\frac{\Omega_{ij}}{A_{ij}}
=
\frac{\Omega_{ij}}{a_i a_j}.
\]
Then the stress matrix $\Omega-c_{ij}A$ has zero $(i,j)$-entry. Hence it is a stress matrix of the graph $G-ij$.
Since $G$ is edge-minimal generically globally rigid, $G-ij$ is not
generically globally rigid in $\mathbb R^d$. By the stress-rank
characterization of generic global rigidity, Theorem~\ref{thm:stress-rank},
no stress matrix of $G-ij$ can have rank $r$. Therefore
\[
\rank(\Omega-c_{ij}A)<r.
\]
Since $\rank(\Omega)=r$, then $\Omega$ has a $r \times r$ submatrix $\Omega^*$ such that $\det(\Omega^*) \neq 0$.
Let $A^*$ be the submatrix obtained from $A$ by keeping the same set of rows and columns as for $\Omega^*$.
Now $\rank(\Omega-c_{ij}A)<r$ implies that $\det(\Omega^*-c_{ij}A^*)=0$.
However, $\det(\Omega^*-tA^*)$ is a linear function in $t$ (since $A^*$ has rank at most $1$) with constant coefficient $\det(\Omega^*) \neq 0$, which thus has at most one value $t$ for which it cancels. 
Hence all the reals $c_{ij}$ are equal, which proves the lemma.
\end{proof}

We are now ready to prove \Cref{thm:main}.

\begin{proof}[Proof of \Cref{thm:main}]
Suppose, for a contradiction, that $G$ contains a copy of $K_{d+2}$.
Let $X\subseteq V$ be the vertex set of such a clique.
Let $p:V\to \mathbb R^d$ be a generic framework of $G$, and set $n=|V|$ and $r=n-d-1$.

By Lemma~\ref{lem:clique-proportionality-full-rank}, there exists
a nonzero real $a_i$, for every $i \in X$, such that for every rank-$r$ stress
matrix $\Omega$ of $(G,p)$, there is a real $\lambda=\lambda(\Omega)$ such that for every $ij\in E(G[X])$ we have
\begin{align*}
    \Omega_{ij}=\lambda a_i a_j.
\end{align*}
We first extend this conclusion to all
stress matrices of $(G,p)$. 
By Lemma~\ref{lem:clique-proportionality-full-rank}, for every rank-$r$
stress matrix $\Omega$ and every two clique edges $ij,kl\in E(G[X])$, we have
\[
\frac{\Omega_{ij}}{a_i a_j}
=
\frac{\Omega_{kl}}{a_k a_l}.
\]
Since the rank-$r$ stress matrices are dense in the stress space by
Lemma~\ref{lem:full-rank-stresses-dense}, the same equality
holds for every stress matrix $\Omega$ of $(G,p)$.
Thus, for every
stress matrix $\Omega$ there exists a real $\lambda=\lambda(\Omega)$
such that for every $ij\in E(G[X])$ we have
\begin{align*}
    \Omega_{ij}=\lambda a_i a_j.
\end{align*}
Since every $a_i$ is nonzero, we have that if a
stress matrix is nonzero on one edge of the clique $G[X]$, then it is
nonzero on every edge of the clique.

We now choose an edge $e\in E(G[X])$.
Since $n\ge d+3$ and $G$ is connected, there is an edge $f\in E(G)\setminus E(G[X])$.
By Theorem~\ref{thm:Rd-connected}, the graph $G$ is $\mathcal R_d$-connected.
Hence there exists an $\mathcal R_d$-circuit $C$ such that $e,f\in C$.

Let $\Omega_C$ be a nonzero stress matrix supported on the circuit $C$ extended by zero outside $C$. Such a stress matrix exists since the rows of the generic rigidity matrix indexed by $C$ are linearly dependent, with not all zero coefficients $\omega_{uv}$ for $uv\in C$ say, and we may define $\Omega_C$ by setting, for $u \neq v$,
$(\Omega_C)_{uv}= \omega_{uv}$ if $uv\in C$, $(\Omega_C)_{uv}= 0$ if $uv\notin C$, and $(\Omega_C)_{uu}=-\sum_{v\ne u}(\Omega_C)_{uv}$. It is then easy to check that $\Omega_C$ is a nonzero stress matrix on the circuit $C$.
Note further that $\Omega_C$ is nonzero on every edge of $C$, as otherwise this would contradict the minimality of $C$. 
In particular,
\[
(\Omega_C)_e\ne 0.
\]
As $e$ is an edge of the clique $G[X]$, it follows by the proportionality property
proved above that for every $ij\in E(G[X])$ we have
\[
(\Omega_C)_{ij}\ne 0
\qquad
\]
Thus every edge of $G[X]$ belongs to the support of $\Omega_C$. Since
$\Omega_C$ is supported on $C$, this implies
\[
E(G[X])\subseteq C.
\]
Since $G[X]\cong K_{d+2}$ is itself an $\mathcal R_d$-circuit, we have that $C=E(G[X])$.
But $f\in C$ and $f\notin E(G[X])$, which is a contradiction.
\end{proof}

\bibliographystyle{abbrv}
\bibliography{bib}

\end{document}